\documentclass[12pt,a4paper]{article}
\usepackage{amssymb}
\usepackage{amsfonts}

\newtheorem{theorem}{Theorem}[section]

\newtheorem{corollary}[theorem]{Corollary}

\newtheorem{definition}[theorem]{Definition}

\newtheorem{lemma}[theorem]{Lemma}

\newtheorem{proposition}[theorem]{Proposition}
\newtheorem{remark}[theorem]{Remark}

\newenvironment{proof}[1][Proof]{\noindent\textbf{#1.} }{\ \rule{0.5em}{0.5em}}

\begin{document}

\title{The Grothendieck property in Marcinkiewicz spaces}
\author{B. de Pagter and F.A. Sukochev \and \textit{Dedicated to the memory
of W.A.J. Luxemburg}}
\date{}
\maketitle

\begin{abstract}
The main pupose of this paper is to fully characterize continuous concave
functions $\psi $ such that the corresponding Marcinkiewicz Banach function
space $M_{\psi }$ is a Grothendieck space.
\end{abstract}

\renewcommand{\thefootnote}{}

\renewcommand{\thefootnote}{}

\section{Introduction}

One of the results obtained by A. Grothendieck in his seminal paper \cite{Gr}%
, is that in the dual space $C\left( K\right) ^{\ast }$ of $C\left( K\right) 
$, where $K$ is an extremally disconnected compact Hausdorff space, any weak-%
$\ast $ convergent sequence is weakly convergent (\cite{Gr}, Th\'{e}or\`{e}%
me 9). This result has motivated to term a Banach space $X$ a \textit{%
Grothendieck space} (or, $X$ is said to have the \textit{Grothendieck
property}) whenever each weak-$\ast $ convergent sequence in the Banach dual 
$X^{\ast }$ is weakly convergent. For example, it was shown by G.L. Seever 
\cite{Se}, Theorem B, that $C\left( K\right) $ is a Grothendieck space for
any compact Hausdorff $F$-space. It should also be noted that it follows
from Grothendieck's result that any $L^{\infty }\left( \mu \right) $, where $%
\mu $ is a $\sigma $-finite measure, is a Grothendieck space (as $L^{\infty
}\left( \mu \right) $ is isometrically isomorphic to $C\left( K\right) $ for
some extremally disconnected compact Hausdorff space $K$). Other examples of
Grothendieck spaces include the Hardy space $H^{\infty }\left( D\right) $
(see \cite{Bou}) and von Neumann algebras (see \cite{Pf}, Corollary 7). For
more information concerning Grothendieck spaces we refer the reader e.g. to
the book \cite{MN}, Section 5.3, and the references given there.

In \cite{L}, H.P. Lotz proved that weak-$L_{p}$ spaces $L_{p,\infty }\left(
\mu \right) $ with $1<p<\infty $, where $\mu $ is a $\sigma $-finite
measure, have the Grothendieck property. Actually, Lotz provided a general
criterion for Banach lattices to be a Grothendieck space (see Theorem \ref%
{Thm03} below), which is then applied to weak-$L_{p}$ spaces.

In this paper we show that Lotz' criterion can be used to exhibit a large
class of Marcinkiewicz spaces $M_{\psi }$ (for a definition, see the
beginning of Section \ref{SectMar}) which are Grothendieck spaces (see
Theorems \ref{Thm04} and \ref{Thm05} below). This class of Marcinkiewicz
spaces includes the weak-$L_{p}$ spaces (see Corollary \ref{Cor02}).

We have chosen to include an "alternative" proof of Lotz' theorem (Theorem %
\ref{Thm03}), based on a general result in Banach lattices (see Proposition %
\ref{Prop01} below), as we find this approach more transparent than the
original proof of Lotz. The reader interested in further details concerning
the Grothendieck property in Marcinkiewicz spaces is referred \cite{PvdB},
where some results from this article are also presented (with different
proofs). It seems to be an interesting open question whether the results
presented in this article also hold for noncommutative Marcinkiewicz spaces.
Finally, we mention that all Banach lattices considered in this paper are
assumed to be real. However, it is easily verified that all results extend
to the complex situation (via complexification). The authors thank Jinghao
Huang for useful discussions of the results presented in this article.

\section{Preliminaries}

In this section we introduce some notation and terminology that will be
used. Given a Banach space $\left( X,\left\Vert \cdot \right\Vert
_{X}\right) $, the dual Banach space is denoted by $\left( X^{\ast
},\left\Vert \cdot \right\Vert _{X^{\ast }}\right) $ and, for $x\in X$ and $%
x^{\ast }\in X^{\ast }$, we write $x^{\ast }\left( x\right) =\left\langle
x,x^{\ast }\right\rangle $. The weak topology in $X$ and the weak-$\ast $
topology in $X^{\ast }$ will be denoted by $\sigma \left( X,X^{\ast }\right) 
$ and $\sigma \left( X^{\ast },X\right) $, respectively.

Let $\left( \Omega ,\Sigma ,\mu \right) $ be a $\sigma $-finite measure
space. The space of all (equivalence classes of) $\mu $-measurable $\mathbb{R%
}$-valued functions on $\Omega $ is denoted by $L_{0}\left( \mu \right) $.
For any $f\in L_{0}\left( \mu \right) $, \textit{the distribution function} $%
d_{\left\vert f\right\vert }:\left[ 0,\infty \right) \rightarrow \left[
0,\infty \right] $ of $\left\vert f\right\vert $ is given by 
\begin{equation}
d_{\left\vert f\right\vert }\left( s\right) =\mu \left( \left\{ x\in \Omega
:\left\vert f\left( x\right) \right\vert >s\right\} \right) ,\ \ \ s\in 
\left[ 0,\infty \right) .  \label{eq16}
\end{equation}%
The function $d_{\left\vert f\right\vert }$ is decreasing and
right-continuous. The \textit{decreasing rearrangement} $f^{\ast }:\left[
0,\infty \right) \rightarrow \left[ 0,\infty \right] $ of $\left\vert
f\right\vert $ is now defined by 
\[
f^{\ast }\left( t\right) =\inf \left\{ s\geq 0:d_{\left\vert f\right\vert
}\left( s\right) \leq t\right\} ,\ \ \ t\in \left[ 0,\infty \right) , 
\]%
and is decreasing and right-continuous. Furthermore, the distribution
function of $f^{\ast }$ (with respect to Lebesgue measure $m$ on $\left[
0,\infty \right) $) is equal to $d_{\left\vert f\right\vert }$ (i.e., $%
\left\vert f\right\vert $ and $f^{\ast }$ are equimeasurable). It should be
observed that if $\mu \left( \Omega \right) <\infty $, then $f^{\ast }\left(
t\right) =0$ for all $t\geq \mu \left( \Omega \right) $ and $f\in
L_{0}\left( \mu \right) $. The notions of distribution function and
decreasing rearrangement are only of interest for functions $f\in
L_{0}\left( \mu \right) $ for which there exists $s_{0}>0$ such that $%
d_{\left\vert f\right\vert }\left( s_{0}\right) <\infty $. The space of all
such functions is sometimes denoted by $S\left( \mu \right) $. If $f\in
S\left( \mu \right) $, then $\lim_{s\rightarrow \infty }d_{\left\vert
f\right\vert }\left( s\right) =0$ and $f^{\ast }\left( t\right) <\infty $
for all $t\in \left( 0,\infty \right) $.

For further properties of decreasing rearrangements we refer the reader to
the books \cite{BS} and \cite{KPS} (see also \cite{Lu}). In particular, we
recall that if the measure $\mu $ is atomless, then 
\begin{equation}
\int_{0}^{t}f^{\ast }\left( s\right) ds=\sup \left\{ \int_{A}\left\vert
f\right\vert d\mu :A\in \Sigma ,\mu \left( A\right) \leq t\right\} ,\ \ \
t\geq 0,\ \ f\in L_{0}\left( \mu \right)  \label{eq03}
\end{equation}%
(see e.g. \cite{BS}, Ch. 2, Proposition 3.3, (b)).

If $\left( \Omega ,\Sigma ,\mu \right) $ and $\left( \tilde{\Omega},\tilde{%
\Sigma},\tilde{\mu}\right) $ are $\sigma $-finite measure spaces and $f\in
L_{0}\left( \mu \right) $, $g\in L_{0}\left( \tilde{\mu}\right) $ are such
that 
\[
\int_{0}^{t}f^{\ast }\left( s\right) ds\leq \int_{0}^{t}g^{\ast }\left(
s\right) ds,\ \ \ t\geq 0, 
\]%
then we say that $f$ is \textit{submajorized} by $g$, which is denoted by $%
f\prec \!\prec g$. The following result may also be deduced via Lemma 2.3 in 
\cite{CDS}. For the reader's convenience, we include an indication of the
proof. The map $D_{2}:L_{0}\left( 0,\infty \right) \rightarrow L_{0}\left(
0,\infty \right) $ is the dilation operator given by $\left( D_{2}f\right)
\left( t\right) =f\left( t/2\right) $, for $t\geq 0$ and $f\in L_{0}\left(
0,\infty \right) $.

\begin{lemma}
\label{Lem07}Let $\left( \Omega ,\Sigma ,\mu \right) $ be a $\sigma $-finite
atomless measure space and let $0<\alpha \leq \infty $. Suppose that $u,v\in
L_{0}\left( \mu \right) ^{+}$ with $u\wedge v=0$ and suppose that $f\in
L_{0}\left( 0,\infty \right) ^{+}$ is such that 
\begin{equation}
\int_{0}^{t}u^{\ast }\left( s\right) ds\leq \int_{0}^{t}f^{\ast }\left(
s\right) ds,\ \ \int_{0}^{t}v^{\ast }\left( s\right) ds\leq
\int_{0}^{t}f^{\ast }\left( s\right) ds,\ \ \ 0\leq t<\alpha .  \label{eq08}
\end{equation}%
Then 
\begin{equation}
\int_{0}^{t}\left( u+v\right) ^{\ast }\left( s\right) ds\leq
\int_{0}^{t}D_{2}f^{\ast }\left( s\right) ds,\ \ \ 0\leq t<\alpha .
\label{eq09}
\end{equation}
\end{lemma}

\begin{proof}
First we observe that if $f_{1},f_{2}\in L_{0}\left( 0,\infty \right) ^{+}$,
then 
\begin{equation}
\int_{0}^{t_{1}}f_{1}^{\ast }\left( s\right) ds+\int_{0}^{t_{2}}f_{2}^{\ast
}\left( s\right) ds\leq \int_{0}^{t_{1}+t_{2}}\left( f_{1}+f_{2}\right)
^{\ast }\left( s\right) ds,\ \ \ t_{1},t_{2}\geq 0,  \label{eq02}
\end{equation}%
(cf. \cite{LSF}, Lemma 3.3.5). Indeed, let $A,B\subseteq \left( 0,\infty
\right) $ be measurable with $m\left( A\right) \leq t_{1}$ and $m\left(
B\right) \leq t_{2}$. Then 
\begin{eqnarray*}
\int_{A}f_{1}dm+\int_{B}f_{2}dm &\leq &\int_{A\cup B}\left(
f_{1}+f_{2}\right) dm\leq \int_{0}^{m\left( A\cup B\right) }\left(
f_{1}+f_{2}\right) ^{\ast }\left( s\right) ds \\
&\leq &\int_{0}^{t_{1}+t_{2}}\left( f_{1}+f_{2}\right) ^{\ast }\left(
s\right) ds,
\end{eqnarray*}%
and hence, (\ref{eq02}) follows via (\ref{eq03}).

Suppose now that $u,v\in L_{0}\left( \mu \right) ^{+}$ with $u\wedge v=0$
and let $f\in L_{0}\left( 0,\infty \right) ^{+}$ be such that (\ref{eq08})
is satisfied. Let $f_{1},f_{2}\in L_{0}\left( 0,\infty \right) ^{+}$ be such
that $f_{1}\wedge f_{2}=0$ and $f_{1}^{\ast }=f_{2}^{\ast }=f^{\ast }$.

Given $0<t<\alpha $, take $A\in \Sigma $ with $\mu \left( A\right) \leq t$.
Define $A_{1}=A\cap \left\{ u>0\right\} $ and $A_{2}=\left\{ v>0\right\} $.
Since $u\wedge v=0$, it follows that $A_{1}\cap A_{2}=\emptyset $ and so, $%
\mu \left( A_{1}\right) +\mu \left( A_{2}\right) \leq \mu \left( A\right) $.
Using (\ref{eq03}) and (\ref{eq08}), in combination with (\ref{eq02}), we
find that 
\begin{eqnarray*}
\int_{A}\left( u+v\right) d\mu &=&\int_{A_{1}}ud\mu +\int_{A_{2}}vd\mu \leq
\int_{0}^{\mu \left( A_{1}\right) }u^{\ast }\left( s\right) ds+\int_{0}^{\mu
\left( A_{2}\right) }v^{\ast }\left( s\right) ds \\
&\leq &\int_{0}^{\mu \left( A_{1}\right) }f_{1}^{\ast }\left( s\right)
ds+\int_{0}^{\mu \left( A_{2}\right) }f_{2}^{\ast }\left( s\right) ds \\
&\leq &\int_{0}^{\mu \left( A_{1}\right) +\mu \left( A_{2}\right) }\left(
f_{1}+f_{2}\right) ^{\ast }\left( s\right) ds\leq \int_{0}^{t}\left(
f_{1}+f_{2}\right) ^{\ast }\left( s\right) ds.
\end{eqnarray*}%
Since this holds for every $A\in \Sigma $ with $\mu \left( A\right) \leq t$,
it follows from (\ref{eq03}) 
\[
\int_{0}^{t}\left( u+v\right) ^{\ast }\left( s\right) ds\leq
\int_{0}^{t}\left( f_{1}+f_{2}\right) ^{\ast }\left( s\right) ds,\ \ \ 0\leq
t<\alpha . 
\]

Using that $f_{1}\wedge f_{2}=0$ and $f_{1}^{\ast }=f_{2}^{\ast }=f^{\ast }$%
, it easily verified that 
\[
\left( f_{1}+f_{2}\right) ^{\ast }\left( s\right) =f^{\ast }\left(
s/2\right) =D_{2}f^{\ast }\left( s\right) ,\ \ \ s>0. 
\]%
Therefore, we may conclude that (\ref{eq09}) holds. \medskip
\end{proof}

Let $\left( \Omega ,\Sigma ,\mu \right) $ be a $\sigma $-finite measure
space. A \textit{Banach function space} over $\left( \Omega ,\Sigma ,\mu
\right) $ is a linear subspace $E$ of $L_{0}\left( \mu \right) $, which is
also an ideal (i.e., $f\in L_{0}\left( \mu \right) $, $g\in E$ and $%
\left\vert f\right\vert \leq \left\vert g\right\vert $ imply that $f\in E$),
equipped with a norm $\left\Vert \cdot \right\Vert _{E}$ satisfying $%
\left\Vert f\right\Vert _{E}\leq \left\Vert g\right\Vert _{F}$ whenever $%
\left\vert f\right\vert \leq \left\vert g\right\vert $ in $E$ and such that $%
\left( E,\left\Vert \cdot \right\Vert _{E}\right) $ is a Banach space. The
notion of Banach function space goes to the Ph.D. Thesis of W.A.J. Luxemburg 
\cite{Lu1}. A concise introduction into the theory of Banach function spaces
may be found in Chapter 15 of the book \cite{Za1}. See also \cite{BS};
however, we like to point out to the reader that in this book a Banach
function space has, by definition, the Fatou property. A Banach function $%
E\subseteq L_{0}\left( \mu \right) $ has the \textit{Fatou property} if it
follows from $0\leq f_{n}\in E$, $f_{n}\uparrow _{n}$ and $%
\sup_{n}\left\Vert f_{n}\right\Vert _{E}<\infty $ that there exists $f\in
E^{+}$ such that $f_{n}\uparrow _{n}f$ and $\left\Vert f\right\Vert
_{E}=\sup_{n}\left\Vert f_{n}\right\Vert _{E}$.

A Banach function space $E$ over $\left( \Omega ,\Sigma ,\mu \right) $ is
called \textit{rearrangement invariant} (or, \textit{symmetric}), if $f\in
L_{0}\left( \mu \right) $, $g\in E$ and $f^{\ast }\leq g^{\ast }$ imply that 
$f\in E$ and $\left\Vert f\right\Vert _{E}\leq \left\Vert g\right\Vert _{E}$%
. For the theory of rearrangement invariant Banach function spaces we refer
the reader to the book \cite{BS} and \cite{KPS} (see also the seminal
article \cite{Lu}). A rearrangement invariant Banach function space $E$ is
called fully symmetric if it possesses the stronger property that $f\in
L_{0}\left( \mu \right) $, $g\in E$, and $f\prec \!\prec g$ imply that $f\in
E$ and $\left\Vert f\right\Vert _{E}\leq \left\Vert g\right\Vert _{E}$.
Every rearrangement invariant Banach function space over an atomless measure
space with the Fatou property is fully symmetric (see e.g. \cite{BS},
Theorem 2.4.6). Any fully symmetric Banach function space is exact
interpolation space between $L_{1}$ and $L_{\infty }$ (see e.g. \cite{BS},
Theorem V.1.17). In particular, any conditional expectation operator is a
contractive projection in such spaces.

\section{Some Banach lattice results}

In this section we obtain some auxiliary results concerning Banach lattices
which will be used in the sequel. Let $\left( E,\left\Vert \cdot \right\Vert
_{E}\right) $ be a (real) Banach lattice with dual Banach lattice $\left(
E^{\ast },\left\Vert \cdot \right\Vert _{E^{\ast }}\right) $. For the
general theory of Banach lattices, we refer to the books \cite{Z} and \cite%
{MN}. First we recall some terminology and notation.

A subset $A$ of $E$ is called \textit{solid} if $\left\vert x\right\vert
\leq \left\vert y\right\vert $, with $x\in E$ and $y\in A$, implies that $%
x\in A$. The \textit{solid hull} of a set $A\subseteq E$ is denoted by ${\rm sol}\left( A\right) $ and is given by 
\[
{\rm sol}\left( A\right) =\left\{ x\in E:\ \exists \ y\in A,\left\vert
x\right\vert \leq \left\vert y\right\vert \right\} . 
\]%
Given a non-empty bounded subset $A\subseteq E$, a Riesz seminorm $\rho
_{A}:E^{\ast }\rightarrow \left[ 0,\infty \right) $ is defined by setting 
\[
\rho _{A}\left( x^{\ast }\right) =\sup \left\{ \left\langle \left\vert
x\right\vert ,\left\vert x^{\ast }\right\vert \right\rangle :x\in A\right\}
,\ \ \ x^{\ast }\in E^{\ast }. 
\]%
It should be noted that $\rho_{A}=\rho_{{\rm sol}\left( A\right) }$.
If $x\in E$, then we write 
\[
\rho_{x}=\rho _{\left\{ x\right\} }=\rho _{\left[ -\left\vert x\right\vert
,\left\vert x\right\vert \right] }, 
\]%
in which case $\rho _{x}\left( x^{\ast }\right) =\left\langle \left\vert
x\right\vert ,\left\vert x^{\ast }\right\vert \right\rangle $, for $x^{\ast
}\in E^{\ast }$. Similarly, if $B$ is a non-empty bounded subset of $E^{\ast
}$, then a Riesz seminorm $\rho _{B}:E\rightarrow \left[ 0,\infty \right) $
is defined by 
\[
\rho _{B}\left( x\right) =\sup \left\{ \left\langle \left\vert x\right\vert
,\left\vert x^{\ast }\right\vert \right\rangle :x^{\ast }\in B\right\} ,\ \
\ x\in E. 
\]%
Note that $\rho _{B}=\rho_{{\rm sol}\left( B\right) }$.

If $\rho $ is a Riesz seminorm on $E$, then we write $B_{\rho }=\left\{ x\in
E:\rho \left( x\right) \leq 1\right\} $. A subset $A$ of $E$ is called 
\textit{almost order bounded with respect to }$\rho $ if for every $%
0<\varepsilon \in \mathbb{R}$ there exists $0\leq u\in E$ such that 
\[
A\subseteq \left[ -u,u\right] +\varepsilon B_{\rho }. 
\]

\begin{remark}
\label{Rem01}It should be noted that if $0\leq y\in \left[ -u,u\right]
+\varepsilon B_{\rho }$, then $y=y\wedge u+\left( y-u\right) ^{+}$ with $%
\rho \left( \left( y-u\right) ^{+}\right) \leq \varepsilon $.
\end{remark}

The following result, which will be used in the proof of Proposition \ref%
{Prop01}, goes back to O. Burkinshaw and P.G. Dodds \cite{BD} (see also \cite%
{MN}, Theorem 2.3.3).

\begin{theorem}
\label{Thm01}Let $E$ be a Banach lattice. For non-empty bounded solid sets $%
A\subseteq E$ and $B\subseteq E^{\ast }$ the following statements are
equivalent.

\begin{enumerate}
\item[(i)] $A$ is almost order bounded with respect to $\rho _{B}$ and $B$
is almost order bounded with respect to $\rho _{A}$.

\item[(ii)] $\rho _{A}\left( x_{n}^{\ast }\right) \rightarrow 0$ as $%
n\rightarrow \infty $ for every disjoint sequence $\left( x_{n}^{\ast
}\right) _{n=1}^{\infty }$ in $B$.

\item[(iii)] $\rho _{B}\left( x_{n}\right) \rightarrow 0$ as $n\rightarrow
\infty $ for every disjoint sequence $\left( x_{n}\right) _{n=1}^{\infty }$
in $A$.
\end{enumerate}
\end{theorem}

For proof of Proposition \ref{Prop01} below, the following three simple
observations will be useful.

\begin{lemma}
\label{Lem04}Let $X$ be a Banach space, let $\left( x_{n}\right) $ be a
sequence in $X$ such that $x_{n}\rightarrow 0$ weakly as $n\rightarrow
\infty $. If $\left( x_{n}^{\ast }\right) _{n=1}^{\infty }$ is a sequence in 
$X^{\ast }$ and $\varepsilon >0$ such that 
\[
\sup_{k\in \mathbb{N}}\left\vert \left\langle x_{n},x_{k}^{\ast
}\right\rangle \right\vert >\varepsilon ,\ \ \ n\in \mathbb{N}, 
\]%
then there exists a subsequence $\left( x_{n_{j}}\right) $ of $\left(
x_{n}\right) $ and a subsequence $\left( x_{k_{j}}^{\ast }\right) $ of $%
\left( x_{k}^{\ast }\right) $ such that $\left\vert \left\langle
x_{n_{j}},x_{k_{j}}^{\ast }\right\rangle \right\vert >\varepsilon $ for all $%
j\in \mathbb{N}$.
\end{lemma}

\begin{proof}
Take $n_{1}=1$. Since $\sup_{k\in \mathbb{N}}\left\vert \left\langle
x_{1},x_{k}^{\ast }\right\rangle \right\vert >\varepsilon $, there exists $%
k_{1}\in \mathbb{N}$ such that $\left\vert \left\langle
x_{1},x_{k_{1}}^{\ast }\right\rangle \right\vert >\varepsilon $. Suppose
that $n_{1}<\cdots <n_{l}$ and $k_{1}<\cdots <k_{l}$ have been constructed
such that $\left\vert \left\langle x_{n_{j}},x_{k_{j}}^{\ast }\right\rangle
\right\vert >\varepsilon $ for $j=1,\ldots ,l$. Since $x_{n}\rightarrow 0$
weakly as $n\rightarrow \infty $, there exists $n_{l+1}>n_{l}$ such that $%
\left\vert \left\langle x_{n_{l+1}},x_{k}^{\ast }\right\rangle \right\vert
<\varepsilon $ for all $1\leq k\leq k_{l}$. Since $\sup_{k\in \mathbb{N}%
}\left\vert \left\langle x_{n_{l+1}},x_{k}^{\ast }\right\rangle \right\vert
>\varepsilon $, there exists $k_{l+1}>k_{l}$ such that $\left\vert
\left\langle x_{n_{l+1}},x_{k_{l+1}}^{\ast }\right\rangle \right\vert
>\varepsilon $. This completes the proof (by induction). \medskip
\end{proof}

From now on, $E$ will be a Banach lattice.

\begin{lemma}
\label{Lem03}If $u\in E^{+}$ and $\left( x_{n}\right) _{n=1}^{\infty }$ is a
disjoint sequence in $\left[ 0,u\right] $, then $x_{n}\rightarrow 0$ weakly
as $n\rightarrow \infty $.
\end{lemma}

\begin{proof}
Given $x^{\ast }\in \left( E^{\ast }\right) ^{+}$ and $N\in \mathbb{N}$, we
have 
\[
\sum\nolimits_{n=1}^{N}\left\langle x_{n},x^{\ast }\right\rangle
=\left\langle \sum\nolimits_{n=1}^{N}x_{n},x^{\ast }\right\rangle
=\left\langle \bigvee\nolimits_{n=1}^{N}x_{n},x^{\ast }\right\rangle \leq
\left\langle u,x^{\ast }\right\rangle . 
\]%
This shows that $\sum\nolimits_{n=1}^{\infty }\left\langle x_{n},x^{\ast
}\right\rangle <\infty $. Hence, $\left\langle x_{n},x^{\ast }\right\rangle
\rightarrow 0$ as $n\rightarrow \infty $ for all $x^{\ast }\in \left(
E^{\ast }\right) ^{+}$. This suffices for a proof. \medskip
\end{proof}

\begin{lemma}
\label{Lem05}Suppose that $u_{0}\in E^{+}$, $0<\delta \in \mathbb{R}$ and $%
x_{1}^{\ast },\ldots ,x_{N}^{\ast }\in \left( E^{\ast }\right) ^{+}$ are
mutually disjoint. If $\left\langle u_{0},x_{n}^{\ast }\right\rangle \geq
\delta $ for all $1\leq n\leq N$ and $\left\{ x_{n}^{\ast }\right\}
_{n=1}^{N}\subseteq \left[ -y^{\ast },y^{\ast }\right] +\frac{1}{2}\delta
B_{\rho _{u_{0}}}$ for some $y^{\ast }\in \left( E^{\ast }\right) ^{+}$,
then $N\leq 2\left\langle u_{0},y^{\ast }\right\rangle /\delta $.
\end{lemma}

\begin{proof}
Using the observation made in Remark \ref{Rem01}, we write 
\[
x_{n}^{\ast }=x_{n}^{\ast }\wedge y^{\ast }+\left( x_{n}^{\ast }-y^{\ast
}\right) ^{+}, 
\]%
where $\left\langle u_{0},\left( x_{n}^{\ast }-y^{\ast }\right)
^{+}\right\rangle =\rho _{u_{0}}\left( \left( x_{n}^{\ast }-y^{\ast }\right)
^{+}\right) \leq \delta /2$ for $1\leq n\leq N$. Then we have 
\[
\delta \leq \left\langle u_{0},x_{n}^{\ast }\right\rangle \leq \left\langle
u_{0},x_{n}^{\ast }\wedge y^{\ast }\right\rangle +\delta /2 
\]%
and so, $\left\langle u_{0},x_{n}^{\ast }\wedge y^{\ast }\right\rangle \geq
\delta /2$ for all $n=1,\ldots N$. Since $\left\{ x_{n}^{\ast }\wedge
y^{\ast }\right\} _{n=1}^{N}$ are mutually disjoint, it follows that 
\begin{eqnarray*}
N\delta /2 &\leq &\sum\nolimits_{n=1}^{N}\left\langle u_{0},x_{n}^{\ast
}\wedge y^{\ast }\right\rangle =\left\langle
u_{0},\sum\nolimits_{n=1}^{N}\left( x_{n}^{\ast }\wedge y^{\ast }\right)
\right\rangle \\
&=&\left\langle u_{0},\bigvee\nolimits_{n=1}^{N}\left( x_{n}^{\ast }\wedge
y^{\ast }\right) \right\rangle \leq \left\langle u_{0},y^{\ast
}\right\rangle .
\end{eqnarray*}%
Consequently, $N\leq 2\left\langle u_{0},y^{\ast }\right\rangle /\delta $.
\medskip
\end{proof}

\begin{corollary}
\label{Cor01}Let $u_{0}\in E^{+}$, $0<\delta \in \mathbb{R}$ and let $%
D\subseteq \left( E^{\ast }\right) ^{+}$ be such that $\left\langle
u_{0},x^{\ast }\right\rangle \geq \delta $ for all $x^{\ast }\in D$. If $D$
is almost order bounded with respect to $\rho _{u_{0}}$, then there exists
an $M\in \mathbb{N}$ such that each disjoint system in $D$ contains at most $%
M$ elements.
\end{corollary}

It will be convenient to introduce the following terminology.

\begin{definition}
A sequence $\left( x_{n}\right) _{n=1}^{\infty }$ in a Banach lattice $E$ is
called \emph{finitely disjoint} if for each $N\in \mathbb{N}$ there exist $%
i_{1}<i_{2}<\cdots <i_{N}$ in $\mathbb{N}$ such that $\left\{
x_{i_{1}},\ldots ,x_{i_{N}}\right\} $ are mutually disjoint.
\end{definition}

\begin{proposition}
\label{Prop01}Let $E$ be a Banach lattice and suppose that $\left(
u_{n}^{\ast }\right) _{n=1}^{\infty }$ is a finitely disjoint sequence in $%
\left( E^{\ast }\right) ^{+}$. If there exist a $u_{0}\in E^{+}$ and a $%
\delta >0$ such that $\left\langle u_{0},u_{n}^{\ast }\right\rangle \geq
\delta $ for all $n$, then there exist a subsequence $\left( u_{n_{k}}^{\ast
}\right) _{k=1}^{\infty }$, a disjoint sequence $\left( v_{k}\right)
_{k=1}^{\infty }$ in $\left[ 0,u_{0}\right] $ and an $\varepsilon >0$ such
that $\left\langle v_{k},u_{n_{k}}^{\ast }\right\rangle \geq \varepsilon $
for all $k\in \mathbb{N}$.
\end{proposition}

\begin{proof}
For $n\in \mathbb{N}$, define the functionals $w_{n}^{\ast }\in E^{\ast }$
by 
\[
w_{n}^{\ast }=\frac{u_{n}^{\ast }}{\left\langle u_{0},u_{n}^{\ast
}\right\rangle }. 
\]%
Let $E_{u_{0}}$ be the principal ideal in $E$ generated by $u_{0}$, i.e., 
\[
E_{u_{0}}=\left\{ x\in E:\left\vert x\right\vert \leq \lambda u_{0}\ {\mbox for some }\lambda \in \mathbb{R}^{+}\right\} , 
\]%
which is a Banach lattice with respect to the order unit norm $\left\Vert
\cdot \right\Vert _{E_{u_{0}}}$, given by 
\[
\left\Vert x\right\Vert _{E_{u_{0}}}=\inf \left\{ \lambda \in \mathbb{R}%
^{+}:\left\vert x\right\vert \leq \lambda u_{0}\right\} . 
\]%
Note that the embedding of $\left( E,\left\Vert \cdot \right\Vert
_{E_{u_{0}}}\right) $ into $\left( E,\left\Vert \cdot \right\Vert
_{E}\right) $ is continuous (in fact, $\left\vert x\right\vert \leq
\left\Vert x\right\Vert _{E_{0}}u_{0}$ implies that $\left\Vert x\right\Vert
_{E}\leq \left\Vert u_{0}\right\Vert _{E}\left\Vert x\right\Vert
_{E_{u_{0}}} $, for $x\in E_{u_{0}}$). Furthermore, it is readily verified
that the restriction map $x^{\ast }\longmapsto x^{\ast }\mid
_{E_{u_{0}}^{\ast }}$, $x^{\ast }\in E^{\ast }$, is a lattice homomorphism
from $E^{\ast }$ into $E_{u_{0}}^{\ast }$.

Let $\varphi _{n}$ be the restriction of $w_{n}^{\ast }$ to $E_{u_{0}}$, for 
$n\in \mathbb{N}$. Since $\varphi _{n}\geq 0$ and $\left\langle
u_{0},\varphi _{n}\right\rangle =1$, it follows that $\left\Vert \varphi
_{n}\right\Vert _{E_{u_{0}}^{\ast }}=1$ for all $n$. Moreover, the sequence $%
\left( \varphi _{n}\right) _{n=1}^{\infty }$ is finitely disjoint in $%
E_{u_{0}}^{\ast }$.

Now define $A=\left[ -u_{0},u_{0}\right] \subseteq E_{u_{0}}$ and $B=
{\rm sol}\left( \left\{ \varphi _{n}\right\} _{n=1}^{\infty }\right)
\subseteq E_{u_{0}}^{\ast }$. It follows from Corollary \ref{Cor01} that the
set $\left\{ \varphi _{n}\right\} _{n=1}^{\infty }$, and hence $B$, is not
almost order bounded in $E_{u_{0}}^{\ast }$ with respect to $\rho
_{u_{0}}=\rho _{A}$. Therefore, condition (i) in Theorem \ref{Thm01} fails
and hence, condition (iii) is not fulfilled. Consequently, there exists a
disjoint sequence $\left( x_{n}\right) _{n=1}^{\infty }$ in $A=\left[
-u_{0},u_{0}\right] $ such that $\rho _{B}\left( x_{n}\right) \nrightarrow 0$
as $n\rightarrow \infty $. Since $\rho _{B}\left( x_{n}\right) =\rho
_{B}\left( \left\vert x_{n}\right\vert \right) $, we may replace $x_{n}$ by $%
\left\vert x_{n}\right\vert $ and so, we may assume that $x_{n}\geq 0$ for
all $n$.

Furthermore, by passing to a subsequence of $\left( x_{n}\right)
_{n=1}^{\infty }$ if necessary, we may assume that there exists an $\gamma
>0 $ such that 
\[
\sup_{n\in \mathbb{N}}\left\langle x_{m},w_{n}^{\ast }\right\rangle
=\sup_{n\in \mathbb{N}}\left\langle x_{m},\varphi _{n}\right\rangle =\rho
_{B}\left( x_{m}\right) >\gamma ,\ \ \ m\in \mathbb{N}. 
\]%
Setting $\varepsilon =\delta \gamma $, it follows that 
\[
\sup_{n\in \mathbb{N}}\left\langle x_{m},u_{n}^{\ast }\right\rangle
=\sup_{n\in \mathbb{N}}\left\langle u_{0},u_{n}^{\ast }\right\rangle
\left\langle x_{m},w_{n}^{\ast }\right\rangle >\varepsilon ,\ \ \ \ m\in 
\mathbb{N}. 
\]%
It follows from Lemma \ref{Lem03} that $x_{m}\rightarrow 0$ weakly as $%
m\rightarrow \infty $ and hence, Lemma \ref{Lem04} implies that there exist
a subsequence $\left( x_{m_{k}}\right) $ of $\left( x_{m}\right) $ and a
subsequence $\left( u_{n_{k}}^{\ast }\right) $ of $\left( u_{n}^{\ast
}\right) $ such that $\left\langle x_{m_{k}},u_{n_{k}}^{\ast }\right\rangle
>\varepsilon $ for all $k\in \mathbb{N}$. Setting $v_{k}=x_{m_{k}}$, the
proof is complete. \medskip
\end{proof}

Next we discuss a condition on a Banach lattice $\left( E,\left\Vert \cdot
\right\Vert _{E}\right) $ implying that the dual Banach lattice $E^{\ast }$
has order continuous norm. Recall that a Banach lattice $E$ is called 
\textit{quasi-uniformly convex} if there exists a constant $0<\alpha <1$
such that $\left\Vert \frac{1}{2}\left( u+v\right) \right\Vert _{E}<\alpha $
whenever $u,v\in E^{+}$ satisfy $u\wedge v=0$ and $\left\Vert u\right\Vert
_{E}=\left\Vert v\right\Vert _{E}=1$ (see \cite{KVP}, Definition 3.31). In
his article \cite{Lo}, G. Ja. Lozanovskii states that for any
quasi-uniformly convex Banach lattice $E$ the dual space $E^{\ast }$ has
order continuous norm. Since we were not able to trace a proof of this
statement in the literature, we include a proof. In this proof we use the
following fact (see e.g. \cite{MN}, Theorem 2.4.14): $E^{\ast }$ has order
continuous norm if and only if every disjoint sequence $\left( x_{n}\right)
_{n=1}^{\infty }$ in $E$ with $\left\Vert x_{n}\right\Vert _{E}\leq 1$, for $%
n\in \mathbb{N}$, satisfies that $x_{n}\rightarrow 0$ as $n\rightarrow
\infty $ with respect to $\sigma \left( E,E^{\ast }\right) $.

\begin{proposition}
\label{Prop02}If $\left( E,\left\Vert \cdot \right\Vert _{E}\right) $ is a
quasi-uniformly convex Banach lattice, then the norm of the the dual Banach
lattice $\left( E^{\ast },\left\Vert \cdot \right\Vert _{E^{\ast }}\right) $
is order continuous.
\end{proposition}

\begin{proof}
Suppose that $E$ is quasi-uniformly convex and let $0<\alpha <1$ be such
that $\left\Vert \frac{1}{2}\left( u+v\right) \right\Vert _{E}<\alpha $
whenever $u,v\in E^{+}$ satisfy $u\wedge v=0$ and $\left\Vert u\right\Vert
_{E}=\left\Vert v\right\Vert _{E}=1$. It is readily verified that this
implies that 
\begin{equation}
\left\Vert \frac{u+v}{2}\right\Vert _{E}\leq \alpha \max \left( \left\Vert
u\right\Vert _{E},\left\Vert v\right\Vert _{E}\right) ,\ \ \ u,v\in E^{+},\
\ u\wedge v=0.  \label{eq01}
\end{equation}

Suppose that $E^{\ast }$ does not have order continuous norm. Then there
exists a disjoint sequence $\left( x_{n}\right) _{n=1}^{\infty }$ in $E$
such that $\left\Vert x_{n}\right\Vert _{E}\leq 1$ for all $n$ and $%
x_{n}\nrightarrow _{n}0$ with respect to $\sigma \left( E,E^{\ast }\right) $%
. So, there exists $\varphi \in E^{\ast }$, with $\left\Vert \varphi
\right\Vert _{E^{\ast }}=1$, such that $\left\langle x_{n},\varphi
\right\rangle \nrightarrow 0$ as $n\rightarrow \infty $. Since $\left\vert
\left\langle x_{n},\varphi \right\rangle \right\vert \leq \left\langle
\left\vert x_{n}\right\vert ,\left\vert \varphi \right\vert \right\rangle $
for all $n$, by replacing $\varphi $ by $\left\vert \varphi \right\vert $
and $x_{n}$ by $\left\vert x_{n}\right\vert $, it may be assumed that $%
\varphi \in \left( E^{\ast }\right) ^{+}$ and that $x_{n}\in E^{+}$ for all $%
n$. Furthermore, by passing, if necessary, to a subsequence of $\left(
x_{n}\right) $, it may be also assumed that there exists $0<\varepsilon \in 
\mathbb{R}$ such that $\left\langle x_{n},\varphi \right\rangle \geq
\varepsilon >0$ for all $n$.

We claim that for each $k\in \mathbb{N}\cup \left\{ 0\right\} $, there
exists a disjoint sequence $\left( y_{n}^{k}\right) _{n=1}^{\infty }$ in $%
E^{+}$ such that $\left\Vert y_{n}^{k}\right\Vert _{E}\leq \alpha ^{k}$ and $%
\left\langle y_{n}^{k},\varphi \right\rangle \geq \varepsilon $ for all $%
n\in \mathbb{N}$.

Indeed, for $k=0$, define $y_{n}^{0}=x_{n}$ (with $\alpha ^{0}=1$). Suppose
now that $k\in \mathbb{N}\cup \left\{ 0\right\} $ is such that the sequence $%
\left( y_{n}^{k}\right) _{n=1}^{\infty }$ has the stated properties.
Defining 
\[
y_{n}^{k+1}=\frac{1}{2}\left( y_{2n-1}^{k}+y_{2n}^{k}\right) ,\ \ \ \ n\in 
\mathbb{N}, 
\]%
it is clear that $\left( y_{n}^{k+1}\right) _{n=1}^{\infty }$ is a disjoint
sequence in $E^{+}$. It follows from (\ref{eq01}) that $\left\Vert
y_{n}^{k+1}\right\Vert _{E}\leq \alpha ^{k+1}$ for all $n$. Furthermore, 
\[
\left\langle y_{n}^{k+1},\varphi \right\rangle =\frac{1}{2}\left(
\left\langle y_{2n-1}^{k},\varphi \right\rangle +\left\langle
y_{2n}^{k},\varphi \right\rangle \right) \geq \varepsilon ,\ \ \ n\in 
\mathbb{N}. 
\]%
This proves the claim.

Let $k\in \mathbb{N}$ be such that $\alpha ^{k}\leq \varepsilon /2$, then,
in particular, 
\[
0<\varepsilon \leq \left\langle y_{1}^{k},\varphi \right\rangle \leq
\left\Vert \varphi \right\Vert _{E^{\ast }}\left\Vert y_{1}^{k}\right\Vert
_{E}\leq \alpha ^{k}\leq \varepsilon /2, 
\]%
which is a contradiction. Therefore, it may be concluded that the norm in $%
E^{\ast }$ is order continuous.\medskip
\end{proof}

\section{A theorem of Lotz}

In this section we exhibit a proof of a theorem of H.P. Lotz \cite{L}
providing sufficient conditions for Banach lattices to have the Grothendieck
property. The following observation will be used.

\begin{lemma}
\label{Lem06}Let $X$ be a Banach space and $\left( x_{n}^{\ast }\right)
_{n=1}^{\infty }$ be a sequence in $X^{\ast }$. If $x_{n}^{\ast
}\nrightarrow 0$ as $n\rightarrow \infty $ with respect to $\sigma \left(
E^{\ast },E^{\ast \ast }\right) $, then there exist a subsequence $\left(
x_{n_{k}}\right) _{k=1}^{\infty }$ of $\left( x_{n}\right) $ and a $\delta
>0 $ such that for each finite set $F\subseteq \mathbb{N}$ there exists $%
x_{F}\in X$ such that $\left\Vert x_{F}\right\Vert _{X}=1$ and $\left\vert
\left\langle x_{F},x_{n_{k}}^{\ast }\right\rangle \right\vert \geq \delta $
for all $k\in F$.

Moreover, if $X=E$ is a Banach lattice and $x_{n}^{\ast }\geq 0$ for all $%
n\in \mathbb{N}$, then we may take $x_{F}\geq 0$.
\end{lemma}

\begin{proof}
Since $x_{n}^{\ast }\nrightarrow 0$ with respect to $\sigma \left( E^{\ast
},E^{\ast \ast }\right) $, there exists $x^{\ast \ast }\in X^{\ast \ast }$
with $\left\Vert x^{\ast \ast }\right\Vert _{X^{\ast \ast }}=1$ such that $%
\left\langle x^{\ast \ast },x_{n}^{\ast }\right\rangle \nrightarrow 0$ as $%
n\rightarrow \infty $. There exists a subsequence $\left( x_{n_{k}}^{\ast
}\right) _{k=1}^{\infty }$ of $\left( x_{n}^{\ast }\right) _{n=1}^{\infty }$%
of such that 
\[
\left\vert \left\langle x^{\ast \ast },x_{n_{k}}^{\ast }\right\rangle
\right\vert \geq \varepsilon >0,\ \ \ k\in \mathbb{N}, 
\]%
for some $\varepsilon >0$.

By Goldstine's theorem (see e.g. \cite{DS}, Theorem V.4.5), the unit ball $%
B_{X}$ is dense in $B_{X^{\ast \ast }}$ with respect to $\sigma \left(
X^{\ast \ast },X^{\ast }\right) $ and so, there exists a net $\left(
x_{\alpha }\right) _{\alpha \in \mathbb{A}}$ in $B_{X}$ such that $x_{\alpha
}\rightarrow _{\alpha }x^{\ast \ast }$ with respect to $\sigma \left(
X^{\ast \ast },X^{\ast }\right) $, i.e., $\left\langle x_{\alpha },x^{\ast
}\right\rangle \rightarrow _{\alpha }\left\langle x^{\ast \ast },x^{\ast
}\right\rangle $ for all $x^{\ast }\in X^{\ast }$. Therefore, if $F\subseteq 
\mathbb{N}$ is a finite set, then there exists $\alpha _{F}\in \mathbb{A}$
such that $\left\vert \left\langle x_{\alpha _{F}},x_{n_{k}}^{\ast
}\right\rangle \right\vert \geq \varepsilon /2$ for all $k\in F$.

Defining $x_{F}=x_{\alpha _{F}}/\left\Vert x_{\alpha _{F}}\right\Vert _{X}$,
it follows that 
\[
\left\vert \left\langle x_{F},x_{n_{k}}^{\ast }\right\rangle \right\vert =%
\frac{1}{\left\Vert x_{\alpha _{F}}\right\Vert _{X}}\left\vert \left\langle
x_{\alpha _{F}},x_{n_{k}}^{\ast }\right\rangle \right\vert \geq \left\vert
\left\langle x_{\alpha _{F}},x_{n_{k}}^{\ast }\right\rangle \right\vert \geq
\varepsilon /2,\ \ \ k\in F, 
\]%
and so we may take $\delta =\varepsilon /2$.

If $X=E$ is a Banach lattice and $x_{n}^{\ast }\geq 0$ for all $n\in \mathbb{%
N}$, then 
\[
\left\langle \left\vert x_{F}\right\vert ,x_{n_{k}}^{\ast }\right\rangle
\geq \left\vert \left\langle x_{F},x_{n_{k}}^{\ast }\right\rangle
\right\vert \geq \varepsilon /2,\ \ \ k\in F, 
\]%
and $\left\Vert \left\vert x_{F}\right\vert \right\Vert _{E}=\left\Vert
x_{F}\right\Vert _{E}=1$. So, in this case, we may replace $x_{F}$ by $%
\left\vert x_{F}\right\vert $. The proof of the lemma is complete. \medskip
\end{proof}

A crucial ingredient in the proof of Lotz' theorem is the following result
due to B. K\H{u}hn \cite{Ku} (see also \cite{MN}, Theorem 5.3.13, (C)).
Recall that a Banach lattice (or, Riesz space) $E$ is said to have the $%
\sigma $\textit{-interpolation property} (or, Property (I), \cite{MN},
Definition 1.1.7) if for any two sequence $\left( x_{n}\right) $ and $\left(
y_{n}\right) $ in $E$ satisfying $x_{n}\leq x_{n+1}\leq y_{n+1}\leq y_{n}$
for all $n\in \mathbb{N}$, there exists $z\in E$ such that $x_{n}\leq z\leq
y_{n}$ for all $n\in \mathbb{N}$. Evidently, every \textit{Dedekind }$\sigma 
$\textit{-complete} Banach lattice (or, Riesz space) has the $\sigma $%
-interpolation property.

\begin{theorem}
\label{Thm02}If $E$ is a Banach lattice with the $\sigma $-interpolation
property, then the following two statements are equivalent.

\begin{enumerate}
\item[(i)] $E$ is a Grothendieck space.

\item[(ii)] $E^{\ast }$ has order continuous norm and $u_{n}^{\ast
}\rightarrow 0$ with respect to $\sigma \left( E^{\ast },E^{\ast \ast
}\right) $ whenever $\left( u_{n}^{\ast }\right) _{n=1}^{\infty }$ is a
disjoint sequence in $\left( E^{\ast }\right) ^{+}$ satisfying $u_{n}^{\ast
}\rightarrow 0$ with respect to $\sigma \left( E^{\ast },E\right) $.
\end{enumerate}
\end{theorem}

The following theorem is due to H.P. Lotz \cite{L}.

\begin{theorem}
\label{Thm03}Let $E$ be a Banach lattice with the $\sigma $-interpolation
property and assume that $E^{\ast }$ has order continuous norm. Suppose that 
$u_{0}\in E^{+}$ and that $\mathcal{G}$ is a collection of positive linear
operators on $E$ such that $T^{\ast }$ is a lattice homomorphism for each $%
T\in \mathcal{G}$, satisfying:

\begin{enumerate}
\item[(a)] for every $x\in E^{+}$ there exists a $T\in \mathcal{G}$ such
that $x\leq \left\Vert x\right\Vert _{E}Tu_{0}$;

\item[(b)] for every disjoint sequence $\left( x_{n}\right) _{n=1}^{\infty }$
in $\left[ 0,u_{0}\right] $ and every sequence $\left( T_{n}\right)
_{n=1}^{\infty }$ in $\mathcal{G}$ there exists $v\in E^{+}$ such that $%
T_{n}x_{n}\leq v$ for all $n$.
\end{enumerate}

Then $E$ is a Grothendieck space.
\end{theorem}

\begin{proof}
Suppose that $E$ is not a Grothendieck space. Since $E$ has the $\sigma $%
-interpolation property and $E^{\ast }$ has order continuous norm, it
follows from Theorem \ref{Thm02} that there exists a disjoint sequence $%
\left( u_{n}^{\ast }\right) _{n=1}^{\infty }$ in $\left( E^{\ast }\right)
^{+}$ such that $u_{n}^{\ast }\rightarrow 0$ with respect to $\sigma \left(
E^{\ast },E\right) $ but, $u_{n}^{\ast }\nrightarrow 0$ with respect to $%
\sigma \left( E^{\ast },E^{\ast \ast }\right) $.

By Lemma \ref{Lem06}, by passing to a subsequence if necessary, we may
assume that there is a $\delta >0$ such that for every finite set $%
F\subseteq \mathbb{N}$ there exists $u_{F}\in E^{+}$ with $\left\Vert
u_{F}\right\Vert _{E}=1$ and $\left\langle u_{F},u_{n}^{\ast }\right\rangle
\geq \delta $ for all $n\in F$.

For each $N\in \mathbb{N}$, define 
\[
F_{N}=\left\{ n\in \mathbb{N}:2^{N-1}\leq n<2^{N}\right\} 
\]%
and let $u_{N}=u_{F_{N}}$ for $N\in \mathbb{N}$.

By hypothesis (a), for each $N\in \mathbb{N}$ there exists an $R_{N}\in 
\mathcal{G}$ such that $u_{N}\leq R_{N}u_{0}$. Define the sequence $\left(
T_{n}\right) _{n=1}^{\infty }$ in $\mathcal{G}$ by setting $T_{n}=R_{N}$
whenever $n\in F_{N}$. Observe that, if $n\in \mathbb{N}$, then 
\[
\left\langle T_{n}u_{0},u_{n}^{\ast }\right\rangle =\left\langle
R_{N}u_{0},u_{n}^{\ast }\right\rangle \geq \left\langle u_{N},u_{n}^{\ast
}\right\rangle \geq \delta , 
\]%
where $N$ is such that $n\in F_{N}$. Furthermore, for each $N\in \mathbb{N}$%
, the system 
\[
\left\{ T_{n}^{\ast }u_{n}^{\ast }:n\in F_{N}\right\} =\left\{ R_{N}^{\ast
}u_{n}^{\ast }:n\in F_{N}\right\} 
\]%
is disjoint, as $R_{N}^{\ast }$ is a lattice homomorphism (by hypothesis).
Therefore, the sequence $\left( T_{n}^{\ast }u_{n}^{\ast }\right) $ is
finitely disjoint and 
\[
\left\langle u_{0},T_{n}^{\ast }u_{n}^{\ast }\right\rangle =\left\langle
T_{n}u_{0},u_{n}^{\ast }\right\rangle \geq \delta ,\ \ \ n\in \mathbb{N}. 
\]%
Hence, by Proposition \ref{Prop01}, there exist a disjoint sequence $\left(
v_{k}\right) _{k=1}^{\infty }$ in $\left[ 0,u_{0}\right] $, a subsequence $%
\left( T_{n_{k}}^{\ast }u_{n_{k}}^{\ast }\right) _{k=1}^{\infty }$ of $%
\left( T_{n}^{\ast }u_{n}^{\ast }\right) $ and an $\varepsilon >0$ such that 
\[
\left\langle v_{k},T_{n_{k}}^{\ast }u_{n_{k}}^{\ast }\right\rangle
>\varepsilon ,\ \ \ k\in \mathbb{N}. 
\]%
By hypothesis (b), there exists $v\in E_{+}$ such that $T_{n_{k}}v_{k}\leq v$
for all $k$. This implies that 
\[
\left\langle v,u_{n_{k}}^{\ast }\right\rangle \geq \left\langle
T_{n_{k}}v_{k},u_{n_{k}}^{\ast }\right\rangle =\left\langle
v_{k},T_{n_{k}}^{\ast }u_{n_{k}}^{\ast }\right\rangle >\varepsilon ,\ \ \
k\in \mathbb{N}, 
\]%
which contradicts the assumption that $u_{n_{k}}^{\ast }\rightarrow 0$ with
respect to $\sigma \left( E^{\ast },E\right) $ as $k\rightarrow \infty $.
This suffices for a proof of the theorem. \medskip
\end{proof}

\section{Marcinkiewicz spaces on $\left( 0,\protect\gamma \right) \label%
{SectMar}$}

Suppose that $0<\gamma \leq \infty $ and let $\psi :\left[ 0,\gamma \right)
\rightarrow \left[ 0,\infty \right) $ be a non-zero increasing concave
function, continuous on $\left( 0,\gamma \right) $ and satisfying $\psi
\left( 0\right) =0$. For $f\in L_{0}\left( 0,\gamma \right) $ define%
\[
\left\Vert f\right\Vert _{M_{\psi }}=\sup_{0<t<\gamma }\frac{1}{\psi \left(
t\right) }\int_{0}^{t}f^{\ast }\left( s\right) ds. 
\]%
The \textit{Marcinkiewicz space} $M_{\psi }\left( 0,\gamma \right) $ is
defined by setting 
\[
M_{\psi }\left( 0,\gamma \right) =\left\{ f\in L_{0}\left( 0,\gamma \right)
:\left\Vert f\right\Vert _{M_{\psi }}<\infty \right\} , 
\]%
which is a rearrangement invariant Banach function space with the Fatou
property. It should be noted that the derivative $\psi ^{\prime }$ (which
exists in all but countably many points) always belongs to $M_{\psi }$ (and
it is easily verified that $\left\Vert \psi ^{\prime }\right\Vert _{M_{\psi
}}=1-\psi \left( 0+\right) /\psi \left( \gamma \right) $). If $\psi \left(
0+\right) =0$, then $\psi \left( t\right) =\int_{0}^{t}\psi ^{\prime }\left(
s\right) ds$, for $0\leq t<\gamma $, and so, in this case the closed unit
ball $B_{M_{\psi }}$ of $M_{\psi }\left( 0,\gamma \right) $ is also given by 
\begin{equation}
B_{M_{\psi }}=\left\{ f\in L_{0}\left( 0,\gamma \right) :f\prec \!\prec \psi
^{\prime }\right\} .  \label{eq06}
\end{equation}

It should be observed that if $1<p<\infty $ and $\psi \left( t\right)
=t^{1-1/p}$, for $0<t<\gamma $, then the corresponding Marcinkiewicz space
equals the weak-$L_{p}$ space $L_{p,\infty }\left( 0,\gamma \right) $.

In this section we are interested in conditions on the increasing concave
function $\psi $ guaranteeing that $M_{\psi }\left( 0,\gamma \right) $ is a
Grothendieck space. The following two conditions on $\psi $ will play an
important role:

\begin{enumerate}
\item[(A)] $\gamma =\infty $ and 
\begin{equation}
\liminf_{t\downarrow 0}\frac{\psi \left( 2t\right) }{\psi \left( t\right) }%
>1,\ \ \ \liminf_{t\rightarrow \infty }\frac{\psi \left( 2t\right) }{\psi
\left( t\right) }>1;  \label{eq04}
\end{equation}

\item[(B)] $\gamma <\infty $ and 
\begin{equation}
\liminf_{t\downarrow 0}\frac{\psi \left( 2t\right) }{\psi \left( t\right) }%
>1.  \label{eq05}
\end{equation}
\end{enumerate}

First we show that these conditions imply that the dual space $M_{\psi
}\left( 0,\gamma \right) ^{\ast }$ has order continuous norm. For the case
that $\gamma <\infty $, the next result is due to Lozanovskii \cite{Lo}.

\begin{proposition}
\label{Prop03}Let $\psi :\left[ 0,\gamma \right) \rightarrow \left[ 0,\infty
\right) $ be a non-zero increasing concave function, continuous on $\left(
0,\gamma \right) $ and satisfying $\psi \left( 0\right) =0$. If either (A)
or (B) holds, then the dual space $M_{\psi }\left( 0,\gamma \right) ^{\ast }$
has order continuous norm.
\end{proposition}

\begin{proof}
(i). Suppose that $\gamma =\infty $ and that (\ref{eq04}) is satisfied. We
claim that there is a \ constant $\beta >1$ such that 
\begin{equation}
\psi \left( 2t\right) /\psi \left( t\right) \geq \beta >1,\ \ \ t\in \left(
0,\infty \right) .  \label{eq07}
\end{equation}%
Indeed, by first condition in (\ref{eq04}), there exist $\delta >0$ and $%
\beta _{1}>1$ such that $\psi \left( 2t\right) /\psi \left( t\right) \geq
\beta _{1}$ for $t\in \left( 0,\delta \right] $. Similarly, the second
condition in (\ref{eq04}) implies that there exist $R>\delta $ and $\beta
_{2}>1$ such that $\psi \left( 2t\right) /\psi \left( t\right) \geq \beta
_{2}$ for $t\in \left[ R,\infty \right) $. Observing that $\psi \left(
2t\right) /\psi \left( t\right) >1$ for $t\in \left[ \delta ,R\right] $ (in
fact, if $t_{0}>0$ is such that $\psi \left( 2t_{0}\right) /\psi \left(
t_{0}\right) =1$, then, using that $\psi $ is concave and increasing, it
follows that $\psi \left( 2t\right) /\psi \left( t\right) =1$ for all $t\geq
t_{0}$ and so the second condition in (\ref{eq04}) cannot be fulfilled), the
claim follows.

Suppose that $u,v\in M_{\psi }\left( 0,\infty \right) ^{+}$ are such that $%
u\wedge v=0$ and $\left\Vert u\right\Vert _{M_{\psi }}=\left\Vert
v\right\Vert _{M_{\psi }}=1$. Then $u\prec \!\prec \psi ^{\prime }$ and $%
v\prec \!\prec \psi ^{\prime }$ (see (\ref{eq06})), and so it follows from
Lemma \ref{Lem07} (with $\alpha =\infty $). that $u+v\prec \!\prec D_{2}\psi
^{\prime }$. Hence, $\left\Vert u+v\right\Vert _{M_{\psi }}\leq \left\Vert
D_{2}\psi ^{\prime }\right\Vert _{M_{\psi }}$. Since 
\[
\left\Vert D_{2}\psi ^{\prime }\right\Vert _{M_{\psi }}=\sup_{t>0}\frac{1}{%
\psi \left( t\right) }\int_{0}^{t}\psi ^{\prime }\left( s/2\right)
ds=2\sup_{t>0}\frac{\psi \left( t/2\right) }{\psi \left( t\right) }, 
\]%
it follows from (\ref{eq07}) that $\left\Vert D_{2}\psi ^{\prime
}\right\Vert _{M_{\psi }}\leq 2\beta ^{-1}$. Consequently, $\left\Vert \frac{%
1}{2}\left( u+v\right) \right\Vert _{M_{\psi }}\leq \beta ^{-1}<1$. This
shows that $M_{\psi }\left( 0,\infty \right) $ is quasi-uniformly convex and
hence, by Proposition \ref{Prop02}, the dual space $M_{\psi }\left( 0,\infty
\right) ^{\ast }$ has order continuous norm.

(ii). Suppose that $\gamma <\infty $ and that (\ref{eq05}) holds. Then there
exist $\delta \in \left( 0,\gamma /2\right) $ and a constant $\beta >1$ such
that 
\begin{equation}
\psi \left( 2t\right) /\psi \left( t\right) \geq \beta >1,\ \ \ t\in \left(
0,\delta \right] .  \label{eq10}
\end{equation}%
Defining 
\[
\left\Vert f\right\Vert _{M_{\psi }}^{\natural }=\sup_{0<t\leq \delta }\frac{%
1}{\psi \left( t\right) }\int_{0}^{t}f^{\ast }\left( s\right) ds,\ \ \ f\in
M_{\psi }\left( 0,\gamma \right) ,
\]%
it is easily verified that $\left\Vert \cdot \right\Vert _{M_{\psi
}}^{\natural }$ is a function norm on $M_{\psi }\left( 0,\gamma \right) $
which is equivalent with $\left\Vert \cdot \right\Vert _{M_{\psi }}$. In
fact, if $f\in M_{\psi }\left( 0,\gamma \right) $, then 
\[
\left\Vert f\right\Vert _{M_{\psi }}^{\natural }\leq \left\Vert f\right\Vert
_{M_{\psi }}\leq \left( 1+\frac{\gamma }{\psi \left( \gamma \right) }\frac{%
\psi \left( \delta \right) }{\delta }\right) \left\Vert f\right\Vert
_{M_{\psi }}^{\natural }.
\]%
Indeed, for $\delta <t<\gamma $ we have 
\begin{eqnarray*}
\frac{1}{\psi \left( t\right) }\int_{0}^{t}f^{\ast }\left( s\right) ds &=&%
\frac{1}{\psi \left( t\right) }\int_{0}^{\delta }f^{\ast }\left( s\right) ds+%
\frac{1}{\psi \left( t\right) }\int_{\delta }^{t}f^{\ast }\left( s\right) ds
\\
&\leq &\frac{1}{\psi \left( \delta \right) }\int_{0}^{\delta }f^{\ast
}\left( s\right) ds+\frac{t}{\psi \left( t\right) }f^{\ast }\left( \delta
\right)  \\
&\leq &\left\Vert f\right\Vert _{M_{\psi }}^{\natural }+\frac{t}{\psi \left(
t\right) }\frac{\psi \left( \delta \right) }{\delta }\left\Vert f\right\Vert
_{M_{\psi }}^{\natural } \\
&\leq &\left( 1+\frac{\gamma }{\psi \left( \gamma \right) }\frac{\psi \left(
\delta \right) }{\delta }\right) \left\Vert f\right\Vert _{M_{\psi
}}^{\natural },
\end{eqnarray*}%
where we used the fact that the function $t\longmapsto \psi \left( t\right)
/t$, for $t\in \left( 0,\gamma \right) $, is decreasing. 

We claim that $\left( M_{\psi }\left( 0,\gamma \right) ,\left\Vert \cdot
\right\Vert _{M_{\psi }}^{\natural }\right) $ is quasi-uniformly convex.
Indeed, suppose that $u,v\in M_{\psi }\left( 0,\gamma \right) ^{+}$ are such
that $u\wedge v=0$ and $\left\Vert u\right\Vert _{M_{\psi }}^{\natural
}=\left\Vert v\right\Vert _{M_{\psi }}^{\natural }=1$. Then 
\[
\int_{0}^{t}u^{\ast }\left( s\right) ds\leq \int_{0}^{t}\psi ^{\prime
}\left( s\right) ds,\ \ \ \int_{0}^{t}u^{\ast }\left( s\right) ds\leq
\int_{0}^{t}\psi ^{\prime }\left( s\right) ds,\ \ \ 0\leq t\leq \delta . 
\]%
By Lemma \ref{Lem07} (with $\alpha =\delta $), this implies that 
\[
\int_{0}^{t}\left( u+v\right) ^{\ast }\left( s\right) ds\leq
\int_{0}^{t}D_{2}\psi ^{\prime }\left( s\right) ds,\ \ \ 0\leq t\leq \delta 
\]%
and hence, $\left\Vert u+v\right\Vert _{M_{\psi }}^{\natural }\leq
\left\Vert D_{2}\psi ^{\prime }\right\Vert _{M_{\psi }}^{\natural }$. Via a
computation similar to the one at the end of the proof of part (a), estimate
(\ref{eq10}) yields that $\left\Vert D_{2}\psi ^{\prime }\right\Vert
_{M_{\psi }}^{\natural }\leq 2\beta ^{-1}$. Hence, $\left\Vert \frac{1}{2}%
\left( u+v\right) \right\Vert _{M_{\psi }^{\natural }}\leq \beta ^{-1}<1$.
This proves the claim.

It now follows from Proposition \ref{Prop02} that the dual space of
\linebreak $\left( M_{\psi }\left( 0,\gamma \right) ,\left\Vert \cdot
\right\Vert _{M_{\psi }}^{\natural }\right) $ has order continuous norm.
Since the norms $\left\Vert \cdot \right\Vert _{M_{\psi }}$ and $\left\Vert
\cdot \right\Vert _{M_{\psi }}^{\natural }$ are equivalent, we may conclude
that the dual space of \linebreak $\left( M_{\psi }\left( 0,\gamma \right)
,\left\Vert \cdot \right\Vert _{M_{\psi }}\right) $ has also order
continuous norm. The proof is complete. \medskip
\end{proof}

\begin{remark}
In the proof of part (ii) of Proposition \ref{Prop03} some care had to be
taken since condition (\ref{eq05}) does not imply that $\left( M_{\psi
}\left( 0,1\right) ,\left\Vert \cdot \right\Vert _{M_{\psi }}\right) $ is
quasi-uniformly convex. By way of example, take $\gamma =1$ and let the
concave function $\psi $ be given by $\psi \left( t\right) =\min \left(
2t,1\right) $, for $t\geq 0$. Defining $u=2\chi _{\left[ 0,1/2\right) }$ and 
$v=2\chi _{\left[ 1/2,1\right) }$, it follows that $\left\Vert u\right\Vert
_{M_{\psi }}=\left\Vert v\right\Vert _{M_{\psi }}=1$ and $\left\Vert \frac{1%
}{2}\left( u+v\right) \right\Vert _{M_{\psi }}=\left\Vert \chi _{\left[
0,1\right) }\right\Vert _{M_{\psi }}=1$. Hence, $\left( M_{\psi }\left(
0,1\right) ,\left\Vert \cdot \right\Vert _{M_{\psi }}\right) $ is \emph{not}
quasi-uniformly convex.
\end{remark}

Actually, the converse of Proposition \ref{Prop03} is also valid as will be
shown in the next proposition. In its proof we make use of so-called \textit{%
symmetric functionals}. A linear functional $0\leq \varphi \in M_{\psi
}\left( 0,\gamma \right) ^{\ast }$ is called symmetric if $0\leq f,g\in
M_{\psi }\left( 0,\gamma \right) $ and $f\prec \!\prec g$ imply that $%
\left\langle f,\varphi \right\rangle \leq \left\langle g,\varphi
\right\rangle $ (see \cite{DPSS}, Definition 2.1). Note that this implies,
in particular, that $\left\langle f,\varphi \right\rangle =\left\langle
f^{\ast },\varphi \right\rangle $, for $0\leq f\in M_{\psi }\left( 0,\gamma
\right) $. A symmetric functional $0\leq \varphi \in M_{\psi }\left(
0,\gamma \right) $ is said to be supported at $0$ if $\varphi =0$ on $%
M_{\psi }\left( 0,\gamma \right) \cap L_{\infty }\left( 0,\gamma \right) $
(equivalently, $\left\langle f^{\ast },\varphi \right\rangle =\left\langle
f^{\ast }\chi _{\left[ 0,s\right) },\varphi \right\rangle $ for all $0<s\leq
\gamma $ and $f\in M_{\psi }\left( 0,\gamma \right) $; see \cite{DPSS},
Lemma 2.9 (a)). In the case that $\gamma <\infty $, the following result may
also be found (with a different proof) in \cite{Lo}.

\begin{proposition}
\label{Prop05}If the dual space $M_{\psi }\left( 0,\gamma \right) $ has
order continuous norm, then either condition (A) or (B) is satisfied.
\end{proposition}

\begin{proof}
Assume that $M_{\psi }\left( 0,\gamma \right) ^{\ast }$ has order continuous
norm. Observe that this implies that $\psi \left( 0+\right) =0$ and, if $%
\gamma =\infty $, that $\lim_{t\rightarrow \infty }\psi \left( t\right)
=\infty $. Indeed, since the K\"{o}the dual $M_{\psi }\left( 0,\gamma
\right) ^{\times }$ is equal to the Lorentz space $\Lambda _{\psi }\left(
0,\gamma \right) $, it follows that $\Lambda _{\psi }\left( 0,\gamma \right) 
$ has order continuous norm, which implies that $\psi \left( 0+\right) =0$
and $\lim_{t\rightarrow \infty }\psi \left( t\right) =\infty $ (if $\gamma
=\infty $).

(i). Suppose that $\gamma =\infty $ and that condition (A) is not satisfied.
Then it follows from \cite{DPSS}, Theorem 3.4, that there exists a non-zero
symmetric functional $0<\varphi \in M_{\psi }\left( 0,\infty \right) ^{\ast
} $. Take $0\leq f\in M_{\psi }\left( 0,\infty \right) $ with $\left\langle
f,\varphi \right\rangle =1$. Let $\left( f_{n}\right) _{n=1}^{\infty }$ be a
disjoint sequence in $M_{\psi }\left( 0,\infty \right) $ such that $%
f_{n}^{\ast }=f^{\ast }$ for all $n$. Then $\left\Vert f_{n}\right\Vert
_{M_{\psi }}=\left\Vert f\right\Vert _{M_{\psi }}$ and $\left\langle
f_{n},\varphi \right\rangle =\left\langle f_{n}^{\ast },\varphi
\right\rangle =\left\langle f^{\ast },\varphi \right\rangle =\left\langle
f,\varphi \right\rangle =1$ for all $n\in \mathbb{N}$. Hence, $\left(
f_{n}\right) _{n=1}^{\infty }$ is a norm bounded disjoint sequence in $%
M_{\psi }\left( 0,\infty \right) $ which does not converge weakly to zero.
This implies that $M_{\psi }\left( 0,\infty \right) ^{\ast }$ does not have
order continuous norm (see \cite{MN}, Theorem 2.4.14), which is a
contradiction.

(ii). Suppose that $\gamma <\infty $ and that condition (B) does not hold.
By \cite{DPSS}, Theorem 3.4 (ii), there exists a non-zero symmetric
functional $0<\psi \in M_{\psi }\left( 0,\gamma \right) ^{\ast }$ supported
at $0$. Take $0\leq f\in M_{\psi }\left( 0,\gamma \right) $ with $%
\left\langle \varphi ,f\right\rangle =1$. For $n\in \mathbb{N}$, define the
intervals $I_{n}=\left[ 2^{-n}\gamma ,2^{-n+1}\gamma \right) $ and let $%
f_{n}\in M_{\psi }\left( 0,\gamma \right) $ be supported on $I_{n}$
satisfying $f_{n}^{\ast }=f^{\ast }\chi _{\left[ 0,2^{-n}\gamma \right) }$.
Then $\left\Vert f_{n}\right\Vert _{M_{\psi }}=\left\Vert f^{\ast }\chi _{%
\left[ 0,2^{-n}\gamma \right) }\right\Vert _{M_{\psi }}\leq \left\Vert
f\right\Vert _{M_{\psi }}$ and $\left\langle f_{n},\varphi \right\rangle
=\left\langle f_{n}^{\ast },\varphi \right\rangle =\left\langle f^{\ast
}\chi _{\left[ 0,2^{-n}\gamma \right) },\varphi \right\rangle =\left\langle
f^{\ast },\varphi \right\rangle =\left\langle f,\varphi \right\rangle =1$
for all $n\in \mathbb{N}$. Hence, $\left( f_{n}\right) _{n=1}^{\infty }$ is
a norm bounded disjoint sequence which does not converge weakly to zero. As
in (i), this yields a contradiction. The proof is complete. \medskip
\end{proof}

Our next objective is to show that conditions (A) or (B) imply that $M_{\psi
}\left( 0,\gamma \right) $ is a Grothendieck space. In the proof, the
following observation will be used. The decreasing function $\Psi :\left(
0,\gamma \right) \rightarrow \left( 0,\infty \right) $, defined by 
\begin{equation}
\Psi \left( t\right) =\frac{\psi \left( t\right) }{t},\ \ \ 0<t<\gamma ,
\label{eq13}
\end{equation}%
will play an important role in the sequel. Part (i) of the next lemma may be
deduced from Lemma 1.4 on p.56 of \cite{KPS}. However, we prefer to include
a direct and perhaps simpler proof.

\begin{lemma}
\label{Lem08}For all $f\in M_{\psi }\left( 0,\gamma \right) $ we have 
\begin{equation}
f^{\ast }\left( t\right) \leq \left\Vert f\right\Vert _{M}\Psi \left(
t\right) ,\ \ \ t\in \left( 0,\gamma \right) .  \label{eq11}
\end{equation}%
Furthermore, if either condition (A) or (B) is fulfilled, then $\Psi \in
M_{\psi }\left( 0,\gamma \right) $.
\end{lemma}

\begin{proof}
Let $f\in M_{\psi }\left( 0,\gamma \right) $. Then 
\[
f^{\ast }\left( t\right) \leq \frac{1}{t}\int_{0}^{t}f^{\ast }\left(
s\right) ds\leq \frac{\psi \left( t\right) }{t}\left\Vert f\right\Vert
_{M_{\psi }},\ \ \ 0<t<\gamma . 
\]%
This proves (\ref{eq11}).

(i). Assume that condition (A) holds. Let $\beta >1$ be such as in (\ref%
{eq07}). The main observation is that 
\[
\int_{t_{1}}^{t_{2}}\frac{\psi \left( s\right) }{s}ds\leq \beta
^{-1}\int_{t_{1}}^{t_{2}}\frac{\psi \left( 2s\right) }{s}ds=\beta
^{-1}\int_{2t_{1}}^{2t_{2}}\frac{\psi \left( s\right) }{s}ds, 
\]%
for all $0<t_{1}<t_{2}$. Given $t>0$, this implies that 
\[
\int_{2^{-n}t}^{2^{-n+1}t}\frac{\psi \left( s\right) }{s}ds\leq \beta
^{-n}\int_{t}^{2t}\frac{\psi \left( s\right) }{s}ds\leq \beta ^{-n}\psi
\left( t\right) ,\ \ \ n\in \mathbb{N}, 
\]%
and hence, 
\[
\int_{0}^{t}\Psi \left( s\right) ds=\sum\nolimits_{n=1}^{\infty
}\int_{2^{-n}t}^{2^{-n+1}t}\frac{\psi \left( s\right) }{s}ds\leq \left(
\beta -1\right) ^{-1}\psi \left( t\right) . 
\]%
Since the function $\Psi $ is decreasing and continuous (and so $\Psi ^{\ast
}=\Psi $ on $\left( 0,\gamma \right) $), this implies that $\Psi \in M_{\psi
}\left( 0,\infty \right) $.

(ii). Assume that condition (B) holds. Let $\beta >1$ and $\delta \in \left(
0,\gamma /2\right) $ be such that (\ref{eq10}) is satisfied. The same
computation as used in (i) now shows that 
\[
\int_{0}^{t}\Psi \left( s\right) ds\leq \left( \beta -1\right) ^{-1}\psi
\left( t\right) ,\ \ \ 0\leq t\leq \delta , 
\]%
which implies that $\Psi \in M_{\psi }\left( 0,\gamma \right) $ (cf. the
proof of (ii) of Proposition \ref{Prop03}). The proof is complete. \medskip
\end{proof}

A bijection $\sigma :\left( 0,\gamma \right) \rightarrow \left( 0,\gamma
\right) $ such that both $\sigma $ and $\sigma ^{-1}$ are measurable and $%
m\left( \sigma ^{-1}\left( A\right) \right) =m\left( A\right) $ for all
(Lebesgue) measurable subsets $A$ of $\left( 0,\gamma \right) $, will be
called an \textit{automorphism} of $\left( 0,\gamma \right) $. If $\sigma $
is an automorphism of $\left( 0,\gamma \right) $, then the map $T_{\sigma
}f=f\circ \sigma $, for $f\in M_{\psi }\left( 0,\gamma \right) $, defines a
linear surjective isometry in any Marcinkiewicz space $M_{\psi }\left(
0,\gamma \right) $ (in fact, $\left( T_{\sigma }f\right) ^{\ast }=f^{\ast }$
for all $f\in S\left( 0,\gamma \right) $). The following lemma is the
counterpart of Lemma 7 in \cite{L}. Recall that $S_{0}\left( 0,\gamma
\right) $ denotes the space of all $f\in S\left( 0,\gamma \right) $ which
satisfy $f^{\ast }\left( \infty \right) =\lim_{t\rightarrow \infty }f^{\ast
}\left( t\right) =0$. Note that, if $\gamma <\infty $, then $S\left(
0,\gamma \right) =S_{0}\left( 0,\gamma \right) $.

\begin{lemma}
\label{Lem09}Let $0<\gamma \leq \infty $.

\begin{enumerate}
\item[(i)] If $g\in S_{0}\left( 0,\gamma \right) ^{+}$ is such that $g\left(
t\right) >0$ a.e. on $\left( 0,\gamma \right) $, then there exists an
automorphism $\sigma $ of $\left( 0,\gamma \right) $ such that $g\leq
2\left( g^{\ast }\circ \sigma \right) $ a.e. on $\left( 0,\gamma \right) $.
\end{enumerate}

Assume now that the concave function $\psi $ on $\left( 0,\infty \right) $
satisfies either condition (A) or (B) and let $\Psi \in M_{\psi }\left(
0,\gamma \right) $ be defined by (\ref{eq13}).

\begin{enumerate}
\item[(ii)] If $f\in M_{\psi }\left( 0,\gamma \right) ^{+}\cap S_{0}\left(
0,\gamma \right) $, then there exists an automorphism $\sigma $ of $\left(
0,\gamma \right) $ such that $f\leq 4\left\Vert f\right\Vert _{M_{\psi
}}T_{\sigma }\Psi $ a.e. on $\left( 0,\gamma \right) $.

\item[(iii)] If $f\in M_{\psi }\left( 0,\gamma \right) ^{+}$, then there
exists an automorphism $\sigma $ of $\left( 0,\gamma \right) $ such that $%
f\leq 5\left\Vert f\right\Vert _{\psi }T_{\sigma }\Psi $.
\end{enumerate}
\end{lemma}

\begin{proof}
(i). This follows from (the proof of) Lemma 3.1 in \cite{L1}.

(ii). First we assume that $\gamma <\infty $ (in which case $M_{\psi }\left(
0,\gamma \right) \subseteq S_{0}\left( 0,\gamma \right) $). Given $0<f\in
M_{\psi }\left( 0,\gamma \right) ^{+}$, take $0<a\in \mathbb{R}$ such that $%
a\left\Vert \Psi \right\Vert _{M_{\psi }}\leq \left\Vert f\right\Vert
_{M_{\psi }}$ and define $g=f+a\Psi $. Then $g$ satisfies the conditions of
(i) and hence, there exists an automorphism $\sigma $ of $\left( 0,\gamma
\right) $ such that $g\leq 2\left( g^{\ast }\circ \sigma \right) $ a.e. on $%
\left( 0,\gamma \right) $. Furthermore, it follows from the definition of $g$
that $g\in M_{\psi }\left( 0,\gamma \right) $ with $\left\Vert g\right\Vert
_{M_{\psi }}\leq 2\left\Vert f\right\Vert _{M_{\psi }}$and so Lemma \ref%
{Lem08} yields that $g^{\ast }\leq 2\left\Vert f\right\Vert _{M_{\psi }}\Psi 
$. Hence, 
\begin{equation}
f\leq g\leq 2T_{\sigma }g^{\ast }\leq 4\left\Vert f\right\Vert _{M_{\psi
}}T_{\sigma }\Psi .  \label{eq12}
\end{equation}

Now we consider the case that $\gamma =\infty $. Two possibilities occur:
either $\Psi \left( \infty \right) =\lim_{t\rightarrow \infty }\Psi \left(
t\right) =0$ or $\Psi \left( \infty \right) >0$. If $\Psi \left( \infty
\right) =0$, then $\Psi \in M_{\psi }\left( 0,\infty \right) \cap
S_{0}\left( 0,\infty \right) $ and an automorphism $\sigma $ satisfying (\ref%
{eq12}) is obtained as in the case that $\gamma <\infty $.

Suppose that $\Psi \left( \infty \right) >0$. Since $f\in S_{0}\left(
0,\infty \right) $, we can choose $0<t_{1}<t_{2}<\cdots \uparrow \infty $
such that $f^{\ast }\left( t_{k}\right) \leq \left( 1/k\right) \Psi \left(
\infty \right) $ for $k\in \mathbb{N}$. Define 
\[
\tilde{\Psi}\left( t\right) =\left\{ 
\begin{array}{ll}
\Psi \left( t\right) & {\mbox if }0<t<t_{1} \\ 
\frac{1}{k}\Psi \left( \infty \right) & \ {\mbox if }t_{k}\leq t<t_{k+1},\ \
k\in \mathbb{N}%
\end{array}%
\right. . 
\]%
Observe that $0\leq \tilde{\Psi}\leq \Psi $, $f^{\ast }\leq \tilde{\Psi}$
and $\tilde{\Psi}\in M_{\psi }\left( 0,\infty \right) ^{+}\cap S_{0}\left(
0,\infty \right) $. By the same argument as used in the case that $\gamma
<\infty $, with $\Psi $ replaced by $\tilde{\Psi}$, it follows that there
exists an automorphism $\sigma $ of $\left( 0,\gamma \right) $ such that $%
f\leq 4\left\Vert f\right\Vert _{\psi }T_{\sigma }\tilde{\Psi}$. Since $%
T_{\sigma }\tilde{\Psi}\leq T_{\sigma }\Psi $, this yields the desired
result.

(iii). If $f\in M_{\psi }\left( 0,\gamma \right) ^{+}\cap S_{0}\left(
0,\gamma \right) $, then this has been proved in (ii). Suppose that $\gamma
=\infty $ and that $f\in M_{\psi }\left( 0,\infty \right) ^{+}$ with $%
f^{\ast }\left( \infty \right) >0$. Write $f=f_{1}+f_{2}$ with $f_{1}=\left(
f-f^{\ast }\left( \infty \right) \chi _{\left( 0,\infty \right) }\right)
^{+} $ and $f_{2}=f\wedge \left( f^{\ast }\left( \infty \right) \chi
_{\left( \left( 0,\infty \right) \right) }\right) $. Using that $f_{1}^{\ast
}=\left( f^{\ast }-f^{\ast }\left( \infty \right) \chi _{\left( 0,\infty
\right) }\right) ^{+}$, it is clear that $f_{1}^{\ast }\left( \infty \right)
=0$ and so $f_{1}\in M_{\psi }\left( 0,\infty \right) ^{+}\cap S_{0}\left(
0,\infty \right) $. By (ii), there exists an automorphism $\sigma $ of $%
\left( 0,\infty \right) $ such that $f_{1}\leq 4\left\Vert f_{1}\right\Vert
_{M_{\psi }}T_{\sigma }\Psi $. Furthermore, $f^{\ast }\left( \infty \right)
\chi _{\left( 0,\infty \right) }\leq f^{\ast }$, and so $\left\Vert f^{\ast
}\left( \infty \right) \chi _{\left( 0,\infty \right) }\right\Vert _{M_{\psi
}}\leq \left\Vert f\right\Vert _{M_{\psi }}$, which implies (by Lemma \ref%
{Lem08}) that $f^{\ast }\left( \infty \right) \chi _{\left( 0,\infty \right)
}\leq \left\Vert f\right\Vert _{M_{\psi }}\Psi $. Hence, 
\[
f_{2}\leq f^{\ast }\left( \infty \right) \chi _{\left( 0,\infty \right)
}=T_{\sigma }\left( f^{\ast }\left( \infty \right) \chi _{\left( 0,\infty
\right) }\right) \leq \left\Vert f\right\Vert _{M_{\psi }}T_{\sigma }\Psi . 
\]%
It follows that 
\[
f=f_{1}+f_{2}\leq 4\left\Vert f_{1}\right\Vert _{M_{\psi }}T_{\sigma }\Psi
+\left\Vert f\right\Vert _{M_{\psi }}T_{\sigma }\Psi \leq 5\left\Vert
f\right\Vert _{M_{\psi }}T_{\sigma }\Psi . 
\]%
The proof is complete. \medskip
\end{proof}

We are now in a position to prove the main result of this section.

\begin{theorem}
\label{Thm04}Let $\psi :\left[ 0,\gamma \right) \rightarrow \left[ 0,\infty
\right) $ be a non-zero concave function, continuous on $\left( 0,\gamma
\right) $ and satisfying $\psi \left( 0\right) =0$. The Marcinkiewicz space $%
M_{\psi }\left( 0,\gamma \right) $ is a Grothendieck space if and only if
either condition (A) or (B) is satisfied.
\end{theorem}

\begin{proof}
If $M_{\psi }\left( 0,\gamma \right) $ is a Grothendieck space, then the
dual space $M_{\psi }\left( 0,\gamma \right) ^{\ast }$ has order continuous
norm (see Theorem \ref{Thm02}) and hence, by Proposition \ref{Prop05},
either condition (A) or (B) is satisfied.

Suppose now that either (A) or (B) holds. Since $M_{\psi }\left( 0,\gamma
\right) $ is Dedekind complete, it certainly has the $\sigma $-interpolation
property. Furthermore, by Proposition \ref{Prop03}, the dual space $M_{\psi
}\left( 0,\gamma \right) ^{\ast }$ has order continuous norm. Therefore, we
can use Lotz' Theorem (Theorem \ref{Thm03}).

Let $u_{0}=5\Psi $, where the function $\Psi $ is defined by (\ref{eq13}).
It follows from Lemma \ref{Lem08} that $u_{0}\in M_{\psi }\left( 0,\gamma
\right) $. Define the collection $\mathcal{G}$ of positive linear operators
on $M_{\psi }\left( 0,\gamma \right) $ by 
\[
\mathcal{G}=\left\{ T_{\sigma }:\sigma {\mbox automorphism of }\left(
0,\gamma \right) \right\} . 
\]%
If $T_{\sigma }\in \mathcal{G}$, then $T_{\sigma }^{\ast }$ and its inverse $%
\left( T_{\sigma }^{\ast }\right) ^{-1}$ are both positive, so $T_{\sigma
}^{\ast }$ is a lattice homomorphism. Furthermore, by Lemma \ref{Lem09}, for
each $f\in M_{\psi }\left( 0,\gamma \right) $ there exists a $T_{\sigma }\in 
\mathcal{G}$ such that $f\leq \left\Vert f\right\Vert _{\psi }T_{\sigma
}u_{0}$. Hence, condition (a) of Theorem \ref{Thm03} is fulfilled. To show
that condition (b) also holds, let $\left( f_{n}\right) _{n=1}^{\infty }$ be
a disjoint sequence in $\left[ 0,u_{0}\right] $ and let $\left( T_{\sigma
_{n}}\right) _{n=1}^{\infty }$ be a sequence in $\mathcal{G}$. Defining the
function $v:\left( 0,\gamma \right) \rightarrow \left[ 0,\infty \right] $ by 
$v\left( t\right) =\sup_{n\geq 1}\left( T_{\sigma _{n}}f_{n}\right) \left(
t\right) $, for $t\in \left( 0,\gamma \right) $, we have to show that $v\in
M_{\psi }\left( 0,\gamma \right) $. For each $k\in \mathbb{N}$, the
distribution functions (see (\ref{eq16}) satisfy 
\[
d_{\bigvee\nolimits_{n=1}^{k}T_{\sigma _{n}}f_{n}}\leq
\sum\nolimits_{n=1}^{k}d_{T_{\sigma
_{n}}f_{n}}=\sum\nolimits_{n=1}^{k}d_{f_{n}}\leq d_{u_{0}}, 
\]%
where the last inequality follows from the fact that $f_{1},\ldots ,f_{n}$
are mutually disjoint in $\left[ 0,u_{0}\right] $. This implies that 
\[
\left( \bigvee\nolimits_{n=1}^{k}T_{\sigma _{n}}f_{n}\right) ^{\ast }\leq
u_{0}^{\ast },\ \ \ k\in \mathbb{N}. 
\]%
Since $\bigvee\nolimits_{n=1}^{k}T_{\sigma _{n}}f_{n}\uparrow _{k}v$, it
follows that $v^{\ast }\leq u_{0}^{\ast }$ and hence, $v\in M_{\psi }\left(
0,\gamma \right) $. It now follows from Theorem \ref{Thm03} that $M_{\psi
}\left( 0,\gamma \right) $ is a Grothendieck space. \medskip
\end{proof}

\section{Marcinkiewicz spaces on general measure spaces}

We start this section with a general observation (Proposition \ref{Prop04}),
which will then be applied to the special case of Marcinkiewicz spaces. Let $%
E\left( 0,\gamma \right) $ be a fully symmetric Banach function space with
the Fatou property on $\left( 0,\gamma \right) $, where $0<\gamma \leq
\infty $. As before, the interval $\left( 0,\gamma \right) $ is equipped
with Lebesgue measure $m$.

Given a $\sigma $-finite measure space $\left( \Omega ,\Sigma ,\mu \right) $
with $\mu \left( \Omega \right) =\gamma $, define 
\begin{equation}
E\left( \mu \right) =\left\{ f\in L_{0}\left( \mu \right) :f^{\ast }\in
E\left( 0,\gamma \right) \right\}  \label{eq14}
\end{equation}%
and 
\begin{equation}
\left\Vert f\right\Vert _{E\left( \mu \right) }=\left\Vert f^{\ast
}\right\Vert _{E\left( 0,\gamma \right) },\ \ \ \ f\in E\left( \mu \right) .
\label{eq15}
\end{equation}%
The following result is well known and easy to prove (see e.g. \cite{BS},
Chapter 2, Theorem 4.9).

\begin{proposition}
The space $\left( E\left( \mu \right) ,\left\Vert \cdot \right\Vert
_{E\left( \mu \right) }\right) $ is a fully symmetric Banach function space
on $\left( X,\Sigma ,\mu \right) $ with the Fatou property.
\end{proposition}

\begin{proposition}
\label{Prop04}If $E\left( 0,\gamma \right) $ has the Grothendieck property,
then $E\left( \mu \right) $ has the Grothendieck property.
\end{proposition}

\begin{proof}
The proof is divided into three steps.

\underline{Step 1.} Suppose that the measure space $\left( \Omega ,\Sigma
,\mu \right) $ is separable and atomless. There exists a measure preserving
Boolean isomorphism $\phi $ from the measure algebra $\Sigma _{\mu }$ of $%
\left( \Omega ,\Sigma ,\mu \right) $ onto the measure algebra $\Sigma _{m}$
of $\lambda $ on $\left( 0,\gamma \right) $. This isomorphism $\phi $
induces a bijective lattice isomorphism $T_{\phi }:L_{0}\left( \mu \right)
\rightarrow L_{0}\left( 0,\gamma \right) $ satisfying $\left( T_{\phi
}f\right) ^{\ast }=f^{\ast }$ for all $f\in L_{0}\left( \mu \right) $. From
the definition of $E\left( \mu \right) $ it follows that the restriction of $%
T_{\phi }$ to $E\left( \mu \right) $ is an isometrical isomorphism from $%
E\left( \mu \right) $ onto $E\left( 0,\gamma \right) $. Consequently, $%
E\left( \mu \right) $ has the Grothendieck property.

\underline{Step 2.} Suppose that the measure space $\left( \Omega ,\Sigma
,\mu \right) $ is separable. Consider the product space $Z=\Omega \times %
\left[ 0,1\right] $ equipped with the product measure $\mu \times m$ on $%
\Sigma \times \Lambda $ (where $\Lambda $ denotes the $\sigma $-algebra of
all Lebesgue measurable subsets of $\left[ 0,1\right] $). It is easily
verified that $\mu \times m$ is a separable atomless measure and so, by Step
1, the space $E\left( \mu \times m\right) $ has the Grothendieck property.

Define 
\[
E_{0}=\left\{ f\otimes \chi _{\left[ 0,1\right] }:f\in E\left( \mu \right)
\right\} . 
\]%
Since $\left( f\otimes \chi _{\left[ 0,1\right] }\right) ^{\ast }=f^{\ast }$
for all $f\in L_{0}\left( \mu \right) $, it follows that $E_{0}$ is also
given by 
\begin{equation}
E_{0}=\left\{ g\in E\left( \mu \right) :g{\mbox is }\Sigma \times \Lambda
_{0}-{\mbox measurable}\right\} ,  \label{eq011}
\end{equation}%
where $\Lambda _{0}=\left\{ \left[ 0,1\right] ,\emptyset \right\} $. The map 
$f\longmapsto f\otimes \mathbf{1}_{\left[ 0,1\right] }$, $f\in E\left( \mu
\right) $ is an isometrical lattice isomorphism from $E\left( \mu \right) $
onto $E_{0}$ (as a subspace of $E\left( \mu \times m\right) $).

The conditional expectation operator 
\[
\mathbb{E}\left( \cdot \mid \Sigma \times \Lambda _{0}\right) :\left(
L^{1}+L^{\infty }\right) \left( \mu \times m\right) \rightarrow \left(
L^{1}+L^{\infty }\right) \left( \mu \times m\right) 
\]%
is given by 
\[
\mathbb{E}\left( g\mid \Sigma \times \Lambda _{0}\right) \left( x,t\right)
=\int_{\left[ 0,1\right] }g\left( x,s\right) ds,\ \ \ g\in \left(
L^{1}+L^{\infty }\right) \left( \mu \times m\right) . 
\]%
Since the Banach function space $E\left( \mu \times m\right) $ is fully
symmetric, $\mathbb{E}\left( \cdot \mid \Sigma \times \Lambda _{0}\right) $
is a positive contractive projection in $E\left( \mu \times m\right) $. It
follows from (\ref{eq011}) that its range is equal to $E_{0}$ and hence, $%
E_{0}$ has the Grothendieck property (being a complemented subspace of a
space with the Grothendieck property; this follows easily via e.g. \cite{MN}%
, Proposition 5.3.10). Since $E\left( \mu \right) $ is isometrically
isomorphic to $E_{0}$, we may conclude that $E\left( \mu \right) $ has the
Grothendieck property.

\underline{Step 3.} Suppose now that $\left( \Omega ,\Sigma ,\mu \right) $
is an arbitrary $\sigma $-finite measure space. It suffices to show that for
any sequence $\left( f_{n}\right) _{n=1}^{\infty }$ in $E\left( \mu \right) $
there exists a closed subspace $F$ of $E\left( \mu \right) $ with the
Grothendieck property such that $f_{n}\in F$ for all $n$ (as follows from 
\cite{MN}, Proposition 5.3.10 in combination with the Eberlein-\u{S}mulian
theorem; see also \cite{L}, Lemma 9).

To this end, suppose that $\left( f_{n}\right) $ is a sequence in $E\left(
\mu \right) $. Let $\left( B_{k}\right) _{k=1}^{\infty }$ be a disjoint
sequence in $\Sigma $ such that $\mu \left( B_{k}\right) <\infty $ for all $%
k $ and $\Omega =\bigcup\nolimits_{k=1}^{\infty }B_{k}$ and let $\left\{
r_{k}\right\} _{k=1}^{\infty }$ be an enumeration of $\mathbb{Q}$. Let $%
\Sigma _{0}$ be the $\sigma $-algebra generated by the sets 
\[
A_{k,n,m}=B_{k}\cap \left\{ x\in \Omega :f_{n}\left( x\right) >r_{m}\right\}
,\ \ \ k,n,m\in \mathbb{N}, 
\]%
and let $\mu _{0}$ be the restriction of $\mu $ to $\Sigma _{0}$. Since $%
\Omega =\bigcup\nolimits_{k,m}A_{k,1,m}$, it is clear that $\left( \Omega
,\Sigma _{0},\mu _{0}\right) $ is $\sigma $-finite. Furthermore, $\left(
\Omega ,\Sigma _{0},\mu _{0}\right) $ is separable (as $\Sigma _{0}$ is
generated by a countable collection). Consequently, by Step 2, the space $%
E\left( \mu _{0}\right) $ has the Grothendieck property. It is also clear
that 
\[
E\left( \mu _{0}\right) =E\left( \mu \right) \cap L_{0}\left( \mu
_{0}\right) =\left\{ f\in E\left( \mu \right) :f{\mbox is }\Sigma _{0}-{\mbox
-measurable}\right\} 
\]%
and so, in particular, $E\left( \mu _{0}\right) $ is a closed subspace of $%
E\left( \mu \right) $.

Finally, for all $n,m\in \mathbb{N}$ we have 
\[
\left\{ x\in \Omega :f_{n}>r_{m}\right\} =\bigcup\nolimits_{k}A_{k,n,m}\in
\Sigma _{0} 
\]%
and so, each $f_{n}$ is $\Sigma _{0}$-measurable. Hence, $f_{n}\in E\left(
\mu _{0}\right) $ for all $n\in \mathbb{N}$. This suffices to complete the
proof of the proposition.\medskip
\end{proof}

Let $\left( \Omega ,\Sigma ,\mu \right) $ be a $\sigma $-finite measure
space and set $\gamma =\mu \left( \Omega \right) $. Suppose that $\psi :%
\left[ 0,\gamma \right) \rightarrow \left[ 0,\infty \right) $ is a non-zero
increasing concave function with $\psi \left( 0\right) =0$. The
corresponding Marcinkiewicz space $\left( M_{\psi }\left( \mu \right)
,\left\Vert \cdot \right\Vert _{M_{\psi }\left( \mu \right) }\right) $ over $%
\left( \Omega ,\Sigma ,\mu \right) $ is defined by 
\[
M_{\psi }\left( \mu \right) =\left\{ f\in L_{0}\left( \mu \right)
:\left\Vert f\right\Vert _{M_{\psi }\left( \mu \right) }<\infty \right\} , 
\]%
where 
\[
\left\Vert f\right\Vert _{M_{\psi }\left( \mu \right) }=\sup_{0<t<\gamma }%
\frac{1}{\psi \left( t\right) }\int_{0}^{t}f^{\ast }\left( s\right) ds,\ \ \
f\in L_{0}\left( \mu \right) . 
\]%
In other words, 
\[
M_{\psi }\left( \mu \right) =\left\{ f\in L_{0}\left( \mu \right) :f^{\ast
}\in M_{\psi }\left( 0,\gamma \right) \right\} ,\ \ \ \left\Vert
f\right\Vert _{M_{\psi }\left( \mu \right) }=\left\Vert f^{\ast }\right\Vert
_{M_{\psi }}, 
\]%
as in (\ref{eq14}) and (\ref{eq15}).

The following theorem is now an immediate consequence of Theorem \ref{Thm04}
and Proposition \ref{Prop04}.

\begin{theorem}
\label{Thm05}If the increasing concave function $\psi $ satisfies either
condition (A) or (B), then the corresponding Marcinkiewicz space $M_{\psi
}\left( \mu \right) $ over $\left( \Omega ,\Sigma ,\mu \right) $ is a
Grothendieck space.
\end{theorem}

It should be noted that Lotz' result concerning weak-$L_{p}$ spaces is
included in the above theorem.

\begin{corollary}
\label{Cor02}For $1<p<\infty $, the weak-$L_{p}$ spaces $L_{p,\infty }\left(
\mu \right) $ are Grothendieck spaces.
\end{corollary}

\begin{proof}
Defining $\psi \left( t\right) =t^{1-1/p}$ for $0<t<\gamma =\mu \left(
\Omega \right) $, we have $M_{\psi }\left( \mu \right) =L_{p,\infty }\left(
\mu \right) $. Observing that $\psi \left( 2t\right) /\psi \left( t\right)
=2^{1-1/p}>1$ for all $t$, the claim follows immediately from Theorem \ref%
{Thm05}. \medskip
\end{proof}

\bigskip

\noindent \texttt{B. de Pagter}

\noindent \texttt{Delft Institute of Applied Mathematics,}

\noindent \texttt{Faculty EEMCS,}

\noindent \texttt{Delft University of Technology, }

\noindent \texttt{P.O. Box 5031, 2600 GA Delft, }

\noindent \texttt{The Netherlands.}

\noindent \texttt{b.depagter@tudelft.nl}

\bigskip

\noindent \texttt{F.A. Sukochev}

\noindent \texttt{School of Mathematics and Statistics,}

\noindent \texttt{University of New South Wales,}

\noindent \texttt{Kensington 2052, NSW, Australia.}

\noindent \texttt{f.sukochev@unsw.edu.au}

\end{document}